\documentclass[twoside,11pt]{amsart}
  \usepackage{amsmath,amsfonts,amssymb,amsthm,mathtools,multirow,tabls,mathrsfs}
  \usepackage[a4paper,margin=1.2in]{geometry}  
  \usepackage[all]{xy}
  \usepackage[utf8]{inputenc}
  \usepackage[T1]{fontenc}
  \usepackage[shortlabels]{enumitem}
  \usepackage[colorlinks,citecolor=magenta,linkcolor=black]{hyperref}

  \numberwithin{equation}{section}

\newtheorem{thm}{Theorem}
\newtheorem{lemma}[thm]{Lemma}
\newtheorem{prop}[thm]{Proposition}

\newtheorem{thmi}{Theorem}     

\newtheorem{question}{Question}

\theoremstyle{definition}

\newcommand{\cL}{\mathscr L}

\newcommand{\cO}{\mathscr O}

\newcommand{\cS}{\mathscr S}

\newcommand{\cY}{\mathscr Y}

\newcommand{\bb}[1]{\mathbf{#1}}

\newcommand{\PP}{\bb{P}}

\newcommand{\ZZ}{\bb{Z}}

\renewcommand{\phi}{\varphi}
\newcommand{\isom}{\simeq}

\newcommand{\ra}{\longrightarrow}    

\DeclareMathOperator{\Def}{Def}

\DeclareMathOperator{\Spec}{Spec}

\newcommand{\cExt}{{\mathscr E}\kern -.5pt xt}
\newcommand{\cHom}{\mathscr{H}\kern -.5pt om}

\newcommand{\linsys}[1]{{\lvert{#1}\rvert}}

  \author[P. Achinger]{Piotr Achinger}
  \address{Institute of Mathematics, Polish Academy of Sciences
    \newline\indent ul. Śniadeckich 8, 00-656 Warsaw, Poland}
  \email{pachinger@impan.pl}

  \author[M. Zdanowicz]{Maciej Zdanowicz}
  \address{\'Ecole Polytechnique F\'ed\'erale de Lausanne, Chair of Algebraic Geometry \newline 
    \indent MA C3 585 (Bâtiment MA), Station 8, CH-1015 Lausanne}
  \email{maciej.zdanowicz@epfl.ch}

  \title[Non-liftable Calabi--Yau]{Non-liftable Calabi--Yau varieties \\ in characteristic $p \geq 5$}
  \date{\today}

\begin{document}
\maketitle


\begin{abstract}
  Based on recent work of Totaro, we provide two simple constructions of non-liftable Calabi--Yau varieties in characteristic $p \geq 5$.  Our examples are of dimension ${2p \text{ and } 2p+1}$.
\end{abstract}


\section{Introduction}

In the following short note, we consider \emph{Calabi--Yau varieties} defined as smooth varieties with trivial canonical bundle and satisfying $H^i(X,\cO_X) = 0$, for $0 < i < \dim X$.  An important feature of Calabi--Yau varieties in characteristic zero is the smoothness of their deformation spaces.  There are many proofs of this property (see \cite{Bogomolov,Tian,Todorov,ZivRan}) based both on analytic and algebraic methods.  The situation turns out to be very different in characteristic $p>0$.  In \cite{Hirokado} Hirokado constructed a projective Calabi--Yau threefold over a field of characteristic three which does not lift to characteristic zero, and therefore its deformation space over the ring of Witt vectors is not smooth.  Examples of non-liftable Calabi--Yau threefolds appeared also in works of Schröer \cite{Schroer}, and Cynk and van Straten \cite{Cynk_vanStraten}, however the constructions worked only for finitely many primes $p$.  More precisely, the examples of Schröer were defined over fields of characteristic two and three, and those of Cynk and van Straten over fields of characteristic three and five (the approach of \cite{Cynk_vanStraten} also led to the construction of non-liftable Calabi--Yau algebraic spaces for special values of $p \leq 1000$).

The goal of the present note is to use the recent results of \cite{Totaro} concerning failure of Kodaira vanishing on Fano varieties to construct non-liftable projective Calabi--Yau varieties over fields of characteristic $p \geq 5$.  Our main contribution is the following theorem combining the results of Proposition~\ref{thm:construction1} and Proposition~\ref{thm:construction2} given below.

\begin{thmi}
For every algebraically closed field of characteristic $p \geq 5$ there exist projective Calabi--Yau varieties of dimension $2p$ and $2p+1$ defined over this field which do not lift formally to any extension of the ring of Witt vectors.
\end{thmi}

\noindent Our examples arise as anticanonical sections and two-to-one coverings of Totaro's Fano varieties violating Kodaira vanishing.  We now recall the necessary properties of those varieties.  From now on, by $k$ we denote an algebraically closed field of characteristic $p \geq 5$, and by $W(k)$ the associated ring of Witt vectors.  Let $N = p+2$, and let ${\pi \colon {\rm Fl}(1,2,N) \to {\rm Gr}(2,N)}$ be the natural projection from the partial flag variety to the Grassmanian of two-dimensional subspaces ${\rm Gr}(2,N)$.  The variety $X$ is defined by the Frobenius pullback of $\pi$, that is, the diagram:
\[
\xymatrix{
  X \ar[r]\ar[d] & {\rm Fl}(1,2,N) \ar[d]^{\pi} \\
  {\rm Gr}(2,N) \ar[r]_(0.45){F} & {\rm Gr}(2,N).
  }
\]

\noindent Since ${\rm Fl}(1,2,N)$ is isomorphic to the projectivization $\PP_{{\rm Gr}(2,N)}(\cS)$ of the tautological vector bundle $\cS$ on ${\rm Gr}(2,N)$, the variety $X$ is in fact isomorphic to $\PP_{{\rm Gr}(2,N)}(F^*\cS)$.  By \cite[Theorem 2.1]{Totaro} we know that there exists a very ample divisor $A$ on $X$ satisfying the following properties:
\begin{enumerate}
  \item $\cO_X(A)$ restricts to $\cO_{\PP^1}(1)$ on the fibres of $\pi$,
  \item $-K_X = 2A$, \label{it1}
  \item $\chi(X,\cO_X(A)) < 0$, \label{it2}
  \item $H^i(X,\cO_X(A)) = 0$ for $i \geq 2$. \label{it3}
\end{enumerate}


\section{Non-liftable Calabi--Yau varieties}

In this section, we provide two examples of non-liftable Calabi--Yau varieties in characteristic $p \geq 5$.  The main technical tool that we use to infer non-liftability is the following standard lemma.  For the sake of completeness we provide a simple proof.

\begin{lemma}[{cf. \cite[page 4]{Totaro}}]\label{lem:non_liftability}
Let $Y$ be a smooth projective variety over $k$ satisfying $H^2(Y,\cO_Y) = 0$.  Assume that there exists an ample invertible sheaf $L$ such that ${\chi(Y,L \otimes \cO_Y(K_Y)) < 0}$.  Then $Y$ does not lift formally to any ramified extension of $W(k)$.
\end{lemma}
\begin{proof}
We assume the contrary, that is, we take a discrete valuation ring $(R,\mathfrak{m})$ with residue field $k$ together with a scheme $\cY_{\rm formal}$ over $\mathrm{Spf}\, R$ and an isomorphism ${\cY_{\rm formal} \otimes k \isom Y}$.  Since $H^2(Y,\cO_Y) = 0$, we see that the ample line bundle $L$ deforms over $\mathrm{Spf}\, R$ and therefore by the Grothendieck algebraization theorem there exists a scheme $\cY$ over $\Spec(R)$ together with a relatively ample line bundle $\cL$.  Using invariance of Euler characteristic, we observe that
\[
\chi(Y,L \otimes \cO_Y(K_Y)) = \chi(\cY_{\eta},\cL_{\eta} \otimes \cO_{\cY_\eta}(K_{\cY_\eta})),
\]
where $\eta$ is the geometric generic point of $\Spec(R)$.  By Kodaira vanishing in characteristic zero we know that the last quantity is equal to $H^0(\cY_{\eta},\cL_{\eta} \otimes \cO_{\cY_\eta}(K_{\cY_\eta}))$, and is therefore positive.  This gives a contradiction, and hence finishes the proof.
\end{proof}

We are now ready to provide our constructions.  The first construction is a smooth divisor inside the anticanonical linear system of variety $X$ described above.  

\begin{prop}\label{thm:construction1}
Let $Y \in \linsys{-K_X}$ be a smooth divisor on $X$.  Then $Y$ is Calabi--Yau variety which does not formally lift to any ramified extension of $W(k)$.
\end{prop}
\begin{proof}
First, we prove that $Y$ is in fact a Calabi--Yau variety.  The condition $K_Y = 0$ follows from adjunction $\cO_Y(K_Y) \isom \cO_X(K_X + Y)_{|Y} = \cO_Y$.  In order to prove that $H^i(Y,\cO_Y) = 0$ for $0 < i < \dim Y$, we use the long exact sequence of cohomology associated with the sequence 
\begin{align}
0 \ra \cO_X(K_X) \ra \cO_X \ra \cO_Y \ra 0. \label{eq1}
\end{align}
Indeed, we see that the sequence 
\[
\ldots \ra H^i(X,\cO_X) \ra H^i(Y,\cO_Y) \ra H^{i+1}(X,\cO_X(K_X)) \ra \ldots
\]
is exact, and the side terms vanish by Kempf vanishing, since $X$ is a homogeneous variety and $0$ is a dominant weight, and Serre duality.

To prove that $Y$ does not lift, we twist \eqref{eq1} by $\cO_X(A)$ and to obtain
\[
0 \ra \cO_X(K_X+A) \ra \cO_X(A) \ra \cO_Y(A) \ra 0.
\]
The bundle $\cO_X(K_X+A)$ restricts to $\cO_{\PP^1}(-1)$ along the fibres of the projection 
\[
\pi \colon {\rm Fl}(1,2,N) \to {\rm Gr}(2,N),
\]
and hence all its cohomology groups vanish.  Consequently, we see that 
\[
\chi(Y,\cO_Y(A)) = \chi(X,\cO_X(A)) < 0.
\]  
This prevents $Y$ from being liftable by Lemma~\ref{lem:non_liftability}.
\end{proof}

\noindent Observe that, taking the Stein factorization $Y \to Y' \to {\rm Gr}(2,N)$ of the natural projection, the scheme $Y$ is a small resolution of a double cover of ${\rm Gr}(2,N)$.  Moreover, it is easy to see that $Y'$ is liftable, confirming an expectation of Langer that every Calabi--Yau variety is a small resolution of a singular liftable Calabi--Yau.

We now proceed to the second construction, which is the $\ZZ/2\ZZ$ cyclic covering of $X$ ramified along a smooth divisor inside $\linsys{-2K_X}$.

\begin{prop}\label{thm:construction2}
Let $D \in \linsys{-2K_X}$ be a smooth divisor and let $Y$ be a $\ZZ/2\ZZ$ cyclic covering ramified along $D$.  Then $Y$ is Calabi--Yau variety which does not formally lift to any ramified extension of $W(k)$.
\end{prop}
\begin{proof}
Let $D \in \lvert -2K_X \rvert$ be a smooth divisor, and let $p \colon Y \to X$ be a $\ZZ/2\ZZ$-covering ramified over $D$.  We claim that $Y$ is a non-liftable Calabi--Yau variety.  First, we observe that the canonical bundle of $Y$ is isomorphic to $p^*\cO_X(K_X - K_X) = \cO_Y$, and that 
\[
H^i(Y,\cO_Y) = H^i(X,\cO_X \oplus \cO_X(K_X)) = H^i(X,\cO_X) \oplus H^{n-i}(X,\cO_X) = 0
\] 
for $0 < i < \dim X$.  To conclude we use the vanishing $H^i(X,\cO_X(K_X+A)) = 0$ and projection formula to see that 
\[
H^i(Y,p^*\cO_X(A)) = H^i(X,\cO_X(A)) \oplus H^i(X,\cO_X(K_X + A)) = H^i(X,\cO_X(A)),
\]
 and hence we can apply Lemma~\ref{lem:non_liftability}.
\end{proof}

We remark that apart from non-liftability the above examples also fail to satisfy the Kodaira vanishing.


\section{Open questions}


It would be interesting to know if similar constructions work for $p = 2,3$, which could lead to low-dimensional examples.  An easy idea to obtain non-liftable threefolds, is to take a complete intersection inside different Fano varieties included in \cite[Theorem 3.1]{Totaro}.  Unfortunately the exhibited ample line bundle satisfy weaker positivity properties and the straightforward arguments does not work.  Moreover, it is a peculiarity of characteristic $p > 0$ geometry that Calabi--Yau varieties could be unirational, and therefore it is natural ask if our varieties satisfy this property.


In \cite{Ekedahl} Ekedahl provides a detailed analysis of deformation spaces of non-liftable Calabi--Yau threefolds constructed by Hirokado and Schröer.  This suggests the following questions concerning deformation theory of Totaro's Fano varieties and our Calabi--Yau varieties.

\begin{question}
Is the deformation space $\Def_{X/W(k)}$ of Totaro's example isomorphic to ${\rm Spf}\, k$?  What do the deformation spaces of our Calabi--Yau varieties look like?
\end{question}

It is worth noting that, since the dimension of varieties considered in this paper is greater than $p-1$, the classical results of \cite{Deligne_Illusie} can not be directly applied to reason about mod $p^2$ liftability of examples of Totaro and ours.  This leads to a simplified version of the above question.

\begin{question}
Do the examples of Totaro or Calabi--Yau varieties constructed in this note lift to Witt vectors of length $2$?
\end{question}

\noindent A positive answer to the above would shed some light on the necessity of assumptions used in proof of Kodaira vanishing included in \emph{loc.cit}..  One possible approach to study the deformation theory of Totaro's variety $X$ would be to take advantage of the structure of a projective bundle, that is, analyze the natural transformation of deformation functors $\Def_{X/W(k)} \to \Def_{F^*\cS}$, where $\cS$ is the rank two tautological bundle on ${\rm Gr}(2,n)$, induced by the pushforward of the relative $\cO_{X}(1)$.

\section*{Acknowledgements}

We would like to thank Burt Totaro, Adrian Langer and Stefan Schröer for helpful discussions.


\bibliographystyle{amsalpha} 
\bibliography{bib.bib}


\end{document}